\documentclass{article}
\usepackage[margin=1in]{geometry}

\AtEndDocument{\bigskip{\footnotesize%
  \textsc{Department of Applied Mathematics, Delhi Technological University, Delhi–110042, India} \par
  \textit{E-mail address:} \texttt{spkumar@dtu.ac.in} \par
  \addvspace{\medskipamount}
  \textsc{Department of Applied Mathematics, Delhi Technological University, Delhi–110042, India} \par  
  \textit{E-mail address:} \texttt{suryagiri456@gmail.com} \par
}}
\usepackage{graphicx}
\usepackage{hyperref}
\usepackage[caption = false]{subfig}
\usepackage{amsthm}
\usepackage{amssymb}
\usepackage{mathrsfs}
\usepackage{mathtools}
\usepackage{enumerate}
\usepackage{amssymb}
\usepackage{wrapfig}
\usepackage{float}
\allowdisplaybreaks

\usepackage{thmtools}
\declaretheorem[numbered=no,
name=Theorem A]{theorem A}

\declaretheorem[numbered=no,
name=Theorem B]{theorem B}

\declaretheorem[numbered=no,
name=Theorem C]{theorem C}

\declaretheorem[numbered=no,
name=Theorem D]{theorem D}

\declaretheorem[numbered=no,
name=Theorem E]{theorem E}

\usepackage{geometry}
\numberwithin{equation}{section}
\DeclareMathOperator{\RE}{Re}

\theoremstyle{plain}
\usepackage{lipsum}

\newtheorem{theorem}{Theorem}[section]
\newtheorem{corollary}[theorem]{Corollary}

\theoremstyle{definition}
\newtheorem{definition}[theorem]{Definition}
\theoremstyle{remark}
\newtheorem{remark}{Remark}[section]

\makeatother

\usepackage{authblk}
\begin{document}
\title{Toeplitz Determinants for a Class of Holomorphic Mappings in Higher Dimensions}
\author{Surya Giri and S. Sivaprasad Kumar}


\date{}


	

\maketitle	
	
\begin{abstract}
   In this paper, we establish the sharp bounds of certain Toeplitz determinants formed over the coefficients of mappings from a class defined on the unit ball of complex Banach space and on the unit polydisc in $\mathbb{C}^n$. Derived bounds provide certain new results for the subclasses of normalized univalent functions and extend some known results in higher dimensions. 
\end{abstract}

\vspace{0.5cm}
\maketitle

\section{Introduction}
      Let $\mathcal{S}$ be the class of analytic univalent functions in the unit disk $\mathbb{U} = \{ z \in \mathbb{C}: \vert z \vert <1 \}$  having the form $g(z) = z +\sum_{n=2}^\infty b_n z^n$ and $\mathcal{K}(\alpha) \subset \mathcal{S}$ denote the class of convex functions of order $\alpha$, $0 \leq \alpha <1.$ A function $g \in \mathcal{K}(\alpha)$ if and only if
      $$ \RE \bigg(1 + \frac{z g''(z)}{g'(z)} \bigg) > \alpha, \quad z\in \mathbb{U}.$$
  For $\alpha =0$, the class $\mathcal{K}(\alpha)$ reduces to the class of convex functions $ \mathcal{K} :=\mathcal{K}(0)$. A function $g \in \mathcal{S}$ is said to be starlike of order $\alpha$ if and only if
  $$ \RE \bigg(\frac{z g'(z)}{g(z)}\bigg) > \alpha, \quad z\in \mathbb{U}.$$
  The class of all starlike functions of order $\alpha$ is denoted by $\mathcal{S}^*(\alpha)$ and let $ \mathcal{S}^* :=\mathcal{S}^*(0)$.
    Recently, Ali et al. \cite{Ali} obtained the bounds of certain Toeplitz determinants whose entries are the Taylor series coefficients of functions in $\mathcal{S}$ and some of its subclasses. For $g(z) = z+ \sum_{n=2}^\infty b_n z^n$, the Toeplitz matrix is given by
\begin{equation*}
     T_{m,n}(g)= \begin{bmatrix}
	b_n & b_{n+1} & \cdots & b_{n+m-1} \\
	b_{n+1} & b_n & \cdots & b_{n+m-2}\\
	\vdots & \vdots & \vdots & \vdots\\
    b_{n+m-1} & b_{n+m-2} & \cdots & b_n\\
	\end{bmatrix}.
\end{equation*}
   In particular, the second order Toeplitz determinant is
\begin{equation}\label{intilb}
    \det{T_{2,2}(g)}=  b_2^2  -  b_3^2
\end{equation}
   and the third order Toeplitz determinant is given by
\begin{equation}\label{T31}
    \det T_{3,1}(g)=
     \begin{vmatrix}
     1 & b_{2} & b_{3} \\
	 b_{2} & 1 & b_{2}\\
	 b_3 & b_{2} & 1\\
     \end{vmatrix}
      = 2 b_2^{2} b_3 - 2 b_2^2-  b_3^2+1.
\end{equation}
    Toeplitz matrices and Toeplitz determinants have various applications in pure as well as in applied mathematics. They occur in a variety of fields including partial differential equations, image processing and differential geometry. For more details and applications, we refer \cite{LHLIM}.

     Numerous articles have recently focused on finding sharp estimates for the Toeplitz and Hermitian-Toeplitz determinants for various classes, but in one dimensional complex plane. Ahuja et al. \cite{Ahuja} established the sharp bounds of $\vert \det{T_{2,2}(g)}\vert$ and $\vert \det T_{3,1}(g)\vert$ for the class $\mathcal{K}$ and its subclasses. For more work in this direction, we refer \cite{Lecko,B.Kowa,SGIRI} and the references cited therein. For the class $\mathcal{K}$ and $\mathcal{K}(\alpha)$, the following bounds are proved in \cite{Ahuja}.
\begin{theorem A}\cite{Ahuja}
    If $f \in \mathcal{K}$, then  $ \vert \det T_{2,2}(f)\vert \leq  2$. The bound is sharp.
\end{theorem A}
\begin{theorem B}\cite{Ahuja}
     If $f \in \mathcal{K}$, then $ \vert \det T_{3,1}(f)\vert \leq 4.$ The bound is sharp.

\end{theorem B}
\begin{theorem C}\cite{Ahuja}
     If $f\in \mathcal{K}(\alpha)$, then the following sharp inequality hold:
       $$ \vert \det T_{2,2}(f)\vert \leq  \frac{2 (1-\alpha)^2 (2 \alpha^2 - 6 \alpha +9)}{9}. $$
\end{theorem C}
\begin{theorem D}\cite{Ahuja}\label{thmD}
    If $f\in \mathcal{K}(\alpha)$ and  $\alpha \in [0,1/2]$, then the following sharp inequality hold:
     $$ \vert \det T_{3,1}(f)\vert \leq  \frac{8 \alpha^4 - 34 \alpha^3 + 71 \alpha^2 - 72 \alpha +36}{9}. $$
\end{theorem D}
    In this paper, we generalize the above results in higher dimensions for a class of holomorphic mappings defined on the unit ball in complex Banach space and on the unit polydisc in $\mathbb{C}^n$. Let $X$ be a complex Banach space with respect to norm $\|\cdot \|$ and  $\mathbb{B} = \{ z \in X: \| z\|<1\}$ be the unit ball. When $X= \mathbb{C}$, $\mathbb{B}$ is denoted by $\mathbb{U}$. Let $ \mathbb{C}^n$ denote the space of $n-$complex variables $z = (z_1, z_2, \cdots, z_n)'$ and $\mathbb{U}^n$ be the Euclidean unit ball in $\mathbb{C}^n$. The boundary and distinguished boundary of $\mathbb{U}^n$ are denoted by $\partial \mathbb{U}^n$ and $\partial_0\mathbb{U}^n$, respectively.

    Let $L(X,Y)$ denote the set of all continuous linear operators from $X$ into a complex Banach space $Y$. For each $z \in X \setminus\{0 \}$, let
     $$ T_z = \{ l_z \in L(X,\mathbb{C}) : l_z(z) = \| z\|, \| l_z \| = 1\}.$$
     By the Hahn-Banach theorem, this set is non-empty.

      By $\mathcal{H}(\Omega,\Omega')$, we denote the set of holomorphic mappings from a domain $\Omega \subseteq X$ into a domain $\Omega' \subseteq Y$ and let $\mathcal{H}(\Omega)=\mathcal{H}(\Omega,X).$ If $g \in \mathcal{H}(\mathbb{B})$ and $z \in \mathbb{B}$, then for each $k = 1,2,\cdots,$ there is a bounded symmetric $k-$linear mapping
    $$D^k g(z) : \prod_{j=1}^k X \rightarrow X ,$$
    called the $k^{th}$ order Fr\'{e}chet derivative of $g$ at $z$ such that
    $$ g(w) = \sum_{k=0}^\infty \frac{1}{k!} D^k g(z) ((w -z)^k )$$
    for all $w$ in some neighborhood of $z$.

    A mapping $g\in \mathcal{H}(\Omega)$ is said to be biholomorphic if $g(\Omega)$ is a domain in $X$ and the inverse $g^{-1}$ exists and is holomorphic on $g(\Omega)$. If the Fr\'{e}chet derivative $D g(z)$ has a bounded inverse for each $z\in \Omega,$ then $g \in \mathcal{H}(\Omega)$ is called locally biholomorphic mapping on $\Omega.$ Analog of the class $\mathcal{S}$, let $\mathcal{S}(\mathbb{B})$ denote the class of biholomorphic mappings $g$ from $\mathbb{B}$ into $X$, satisfying $g(0)=0 $ and $Dg(0)=I$, where $I$ represents the linear identity operator from $X$ into $X$. It is easily seen that $\mathcal{S}(\mathbb{B})$ is not a normal family when the dimension is greater than one \cite{Cart}. A mapping $g \in \mathcal{S}(\mathbb{B})$ is said to be starlike if $g(\mathbb{B})$ is starlike with respect to the origin. Also, a mapping $g \in \mathcal{S}(\mathbb{B})$ is said to be convex if $g(\mathbb{B})$ is convex.

   On a bounded circular domain $\Omega \subset \mathbb{C}^n$, the first and the $m^{th}$ Fr\'{e}chet derivative of a holomorphic mapping $g : \Omega \rightarrow X$  are written by
    $ D g(z)$ and $D^m g(z) (a^{m-1},\cdot)$ respectively. The matrix representations are
\begin{align*}
    D g(z) &= \bigg(\frac{\partial g_j}{\partial z_k} \bigg)_{1 \leq j, k \leq n}, \\
    D^m g(z)(a^{m-1}, \cdot) &= \bigg( \sum_{p_1,p_2, \cdots, p_{m-1}=1}^n  \frac{ \partial^m g_j (z)}{\partial z_k \partial z_{p_1} \cdots \partial z_{p_{m-1}}} a_{p_1} \cdots a_{p_{m-1}}   \bigg)_{1 \leq j,k \leq n},
\end{align*}
   where $g(z) = (g_1(z), g_2(z), \cdots g_n(z))'$ and $ a= (a_1, a_2, \cdots a_n)'\in \mathbb{C}^n.$
   \\

    Liu and Liu \cite{Liu} defined the following class:
\begin{definition}\cite{Liu}\label{defn1C}
    Suppose $\alpha \in [0,1)$ and $g: \mathbb{B} \rightarrow X$ is a normalized locally biholomorphic mapping. If
    $$ \RE \left\{l_z [(D g(z))^{-1} ( D^2 g(z) (z^2) + D g(z) (z) )]  \right\} \geq \alpha \| z \|, \quad  l_z \in T_z, \;z \in \mathbb{B}\setminus \{0 \},$$
    then $f$ is called a quasi convex mapping of type $B$ and order $\alpha$ on $\mathbb{B}$.

    If $\mathbb{B} = \mathbb{U}^n$ and $X = \mathbb{C}^n$, then the above condition reduces to
     $$ \left\vert \frac{q_k(z)}{z_k} - \frac{1}{2 \alpha} \right\vert < \frac{1}{2\alpha}, \quad \forall z \in \mathbb{U}^n \setminus \{0 \},$$
   where $$q(z) = ( q_1(z), q_2(z), \cdots , q_n(z))' = (D g(z))^{-1} ( D^2 g(z) (z^2) + D g(z) (z) )$$
   is a column vector in $\mathbb{C}^n$ and $k$ satisfies $$\vert z_k \vert = \| z \| = \max_{1 \leq j \leq n} \{ \vert z_j \vert \}.$$
   For $\mathbb{B} = \mathbb{U}$ and $X = \mathbb{C}$,  the relation is equivalent to
   $$ \RE \bigg( 1 + \frac{ z g''(z)}{g'(z)} \bigg) > \alpha, \quad z \in \mathbb{U}. $$
   Let $\mathcal{K}_\alpha(\mathbb{B})$ denote the class of quasi convex mappings of type $B$ and order $\alpha.$
\end{definition}
    When $\alpha =0$, Definition \ref{defn1C} is the definition of quasi convex mapping of type $B$  introduced by Roper and Suffridge \cite{Rope}.
\begin{definition}
     Let $\Phi: \mathbb{U} \rightarrow \mathbb{C}$ be a biholomorphic function such that $\RE \Phi(z) >0$, $\Phi(0)=1$, $\Phi'(0)>0$ and $\Phi''(0) \in \mathbb{R}$. Let $\mathcal{M}_\Phi$ be the class of mappings given by
     $$ \mathcal{M}_\Phi = \bigg\{ p \in \mathcal{H}(\mathbb{B}) : p(0)= Dp(0)=I, \; \frac{l_z (p(z))}{\|z\|} \in \Phi(\mathbb{U}),\; z\in \mathbb{B}\setminus\{0\}, l_z \in T_z \bigg\}.  $$
     In case of $\mathbb{B} = \mathbb{U}^n$ and $X = \mathbb{C}^n$, we have
     $$ \mathcal{M}_\Phi = \bigg\{ p \in \mathcal{H}(\mathbb{B}) : p(0)= Dp(0)=I, \; \frac{ p_j (z)}{z_j} \in \Phi(\mathbb{U}), \; z\in \mathbb{U}^n\setminus\{0\} \bigg\},  $$
     where $p(z) = (p_1(z), p_2(z), \cdots, p_n(z))'$ is a column vector in $\mathbb{C}^n$ and $j$ satisfies $\vert z_j \vert = \| z\| = \max_{1 \leq k \leq n}\{ \vert z_k \vert \}$.
\end{definition}
    In 1999, Roper and Suffridge \cite{Rope} gave a sufficient condition for a normalized biholomorphic convex mapping on the Euclidean unit ball in $\mathbb{C}^n$. Later, Zhu \cite{Zhu} provided a brief proof of this theorem. Xu et al. \cite{XuC} obtained the sharp bounds of Fekete-Szeg\"{o} inequality for the class of quasi-convex mappings of type $B$ and order $\alpha$ defined on the unit ball $\mathbb{B}$ and on the unit polydisc in $\mathbb{C}^n$.  Liu and Liu \cite{LiuLiu2} derived the sharp estimates of all homogenous expansions for a subclass of holomorphic mappings of  quasi-convex mappings of type $B$ and order $\alpha$ in higher dimensions.
     Contrary to the coefficient inequalities for many subclasses of $\mathcal{S}$, only few are known for homogeneous expansions for subclasses of biholomorphic mappings in the case of several complex variables \cite{Grah2,Ham,Ham2,Kohr,Grah,Xuetal}.

     In case of one complex variable, many coefficient problems are studied for the class $\mathcal{K}$ such as Theorems A-D. A natural question arises that how to retain these results in higher dimensions. Providing an answer to this question is the aim of this study.

\section{Main Results}
   The main results of the paper are stated and proved in this section.
\begin{theorem}\label{thm1C}
   Let  $g \in \mathcal{H}(\mathbb{B}, \mathbb{C})$ with  $g(0)=1$, $g(z) \neq 0$, $z\in \mathbb{B}$ and suppose that $G(z) =  z g(z) $. If $(D G(z))^{-1} ( D^2 G(z) (z^2) + D G(z) (z) )\in \mathcal{M}_\Phi$ such that $\Phi$ satisfies
   $$\lvert \Phi''(0) + 2 (\Phi'(0))^2 \rvert \geq 2  \Phi'(0) > 0, $$
   then
    $$ \bigg\vert \bigg( \frac{ l_z (D^2 G(0) (z^2))}{2! \vert\vert z \vert\vert^2} \bigg)^2 - \bigg(\frac{ l_z (D^3 G(0) (z^3))}{3! \vert\vert z \vert\vert^3} \bigg)^2 \bigg\vert \leq \frac{( \Phi'(0) )^2}{4} + \frac{(\Phi'(0))^2}{36}\bigg( \frac{1}{2} \frac{\Phi''(0)}{\Phi'(0)} +  \Phi'(0) \bigg)^2 .$$
   The bound is sharp.
\end{theorem}
\begin{proof}
    Fix $z \in X\setminus \{0\}$ and let $h : \mathbb{U} \rightarrow \mathbb{C}$ be defined by
\begin{equation*}
    h(\zeta) =
\left\{
\begin{aligned}
\begin{array}{ll}
     \dfrac{l_z ( (D G ( \zeta z_0) )^{-1} ( D^2 G(\zeta z_0) ( (\zeta z_0)^2 )  + D G( \zeta z_0) \zeta z_0 ) )}{\zeta }, & \zeta \neq 0, \\
     1, & \zeta =0,
\end{array}
\end{aligned}
\right.
\end{equation*}
    where $z_0 = \frac{z}{\| z\|}$. Then $h \in \mathcal{H}(\mathbb{U})$ and $h(0) = \Phi(0) =1$.  Since $(D G(z))^{-1} ( D^2 G(z) (z^2) + D G(z) (z) )\in \mathcal{M}_\Phi$, therefore, we have
\begin{align*}
    h(\zeta) &= \frac{l_z ( (D G ( \zeta z_0) )^{-1} ( D^2 G(\zeta z_0) ( (\zeta z_0)^2 )  + D G( \zeta z_0) \zeta z_0 ) )}{\zeta } \\
             &= \frac{l_{z{_0}} ( (D G ( \zeta z_0) )^{-1} ( D^2 G(\zeta z_0) ( (\zeta z_0)^2 )  + D G( \zeta z_0) \zeta z_0 ) )}{\zeta } \\
             &=  \frac{l_{\zeta z{_0}} ( (D G ( \zeta z_0) )^{-1} ( D^2 G(\zeta z_0) ( (\zeta z_0)^2 )  + D G( \zeta z_0) \zeta z_0 ) )}{\| \zeta z_0 \| } \in \Phi(\mathbb{U}), \;\; \zeta \in \mathbb{U}.
\end{align*}
    Applying a similar method used in \cite[Theorem 7.1.14]{GraKoh}, we get
\begin{equation}\label{11}
     (D G(z))^{-1} =   \frac{1}{g(z)} \bigg( I - \frac{ \frac{z D g(z)}{g(z)}}{1 + \frac{D g(z) z}{g(z)}} \bigg).
\end{equation}
    A simple computation using the fact $G(z) = z g(z)$ yields
\begin{equation}\label{111}
     D^2 G(z) (z^2) + D G(z) (z) = ( D^2 g(z) ( z^2 ) +3 D g(z)(z) +g(z) )z.
\end{equation}
    By using (\ref{11}) and (\ref{111}), it follows
\begin{equation}\label{useC}
   (D G(z) )^{-1} ( D^2 G(z) (z^2) + D G(z) (z) ) = \frac{ D^2 g(z) (z^2) + 3 D g(z)(z) + g(z) }{ g(z)  + D g(z) (z)} z.
\end{equation}
    Consequently
\begin{equation}\label{eqn1}
     {l_z ( (D G(z) )^{-1} ( D^2 G(z) (z^2) + D G(z) (z) ) )} =  \frac{ D^2 g(z) (z^2) + 3 D g(z)(z) + g(z) }{ g(z)  + D g(z) (z)}\| z \|.
\end{equation}
    Using (\ref{eqn1}), we obtain
\begin{align*}
     h(\zeta) &= \frac{l_{\zeta z{_0}} ( (D G ( \zeta z_0) )^{-1} ( D^2 G(\zeta z_0) ( (\zeta z_0)^2 )  + D G( \zeta z_0) \zeta z_0 ) )}{\| \zeta z_0 \| } \\
              &=  \frac{ D^2 g(\zeta z_0) ((\zeta z_0)^2 ) + 3 D g(\zeta z_0)(\zeta z_0) + g(\zeta z_0) }{ g(\zeta z_0)  + D g(\zeta z_0) ( \zeta z_0)} .
\end{align*}
     Equivalently, we can write
     $$ h(\zeta) ( g(\zeta z_0)  + D g(\zeta z_0) ( \zeta z_0) ) =  D^2 g(\zeta z_0) ((\zeta z_0)^2 ) + 3 D g(\zeta z_0)(\zeta z_0) + g(\zeta z_0). $$
     The series expansion in terms of $\zeta$ gives
\begin{align*}
   \bigg(1 +  h'(0) \zeta  + \frac{h''(0)}{2} \zeta^2 + \cdots \bigg)  \bigg( 1 & +2 Dg(0)(z_0) \zeta + \frac{3  Dg(0)(z_{0}^2)}{2} \zeta^2 + \cdots \bigg)  \\
        & = 1 +4  Dg(0)(z_0) \zeta + \frac{ 9 Dg(0)(z_{0}^2)}{2} \zeta^2 + \cdots .
\end{align*}
     Comparison of the homogenous expansions of either sides of the above equality provide $ h'(0) = 2 D g(0)(z_0).$
    That is
\begin{equation}\label{eqn2}
     h'(0) \|z \| =  2 D g(0)(z) .
\end{equation}
    Also, we have
    $$ \frac{D^2 G(0) (z^2)}{2!} = D g(0)(z) z,  $$
    which gives
    $$  \frac{ l_z ( D^2 G(0) (z^2)) }{2!} = D g(0)(z) \| z\|. $$
   Now, using $\vert h'(0) \vert \leq \Phi'(0) $ with (\ref{eqn2}), we obtain
\begin{equation}\label{a2BC}
    \bigg\vert \frac{ l_z ( D^2 G(0) (z^2)) }{2! \| z\|^2} \bigg\vert \leq   \frac{\Phi'(0)}{2}.
\end{equation}
    Moreover, for $\lambda \in \mathbb{C}$, Xu et al. \cite[Theorem 3.1]{XuC} proved that
\begin{equation}\label{FSBC}
\begin{aligned}
\left.
\begin{array}{ll}
   \bigg\vert  &\dfrac{ l_z (D^3 G(0) (z^3))}{3! \vert\vert z \vert\vert^3} -  \lambda \bigg(\dfrac{ l_z (D^2 G(0) (z^2))}{2! \vert\vert z \vert\vert^2} \bigg)^2 \bigg\vert \\
   &\leq \dfrac{\vert \Phi'(0) \vert}{6} \max \left\{ 1,  \left\lvert \dfrac{1}{2} \dfrac{\Phi''(0)}{\Phi'(0)} + \bigg(1 - \dfrac{3}{2} \lambda \bigg)  \Phi'(0) \right\rvert \right\} , \;\;  z \in \mathbb{B}\setminus \{0 \}.
\end{array}
\right\}
\end{aligned}
\end{equation}
    Since $ \lvert \Phi''(0) + 2 (\Phi'(0))^2 \rvert \geq 2  \Phi'(0)$, therefore the above inequality gives
\begin{equation}\label{a3BC}
     \bigg\vert  \dfrac{ l_z (D^3 G(0) (z^3))}{3! \vert\vert z \vert\vert^3} \bigg\vert \leq \dfrac{ \Phi'(0)}{6}   \left( \dfrac{1}{2} \dfrac{\Phi''(0)}{\Phi'(0)} +  \Phi'(0) \right).
\end{equation}
     Also, note that
\begin{align*}
     \bigg\vert \bigg( \dfrac{ l_z (D^3 G(0) (z^3))}{3! \vert\vert z \vert\vert^3} \bigg)^2 &-    \bigg(   \frac{ l_z ( D^2 G(0) (z^2))  }{2! \| z\|^2}  \bigg)^2 \bigg\vert \\
      & \leq  \bigg\vert  \dfrac{ l_z (D^3 G(0) (z^3))}{3! \vert\vert z \vert\vert^3} \bigg\vert^2  +   \bigg\vert \frac{ l_z ( D^2 G(0) (z^2)) }{2! \| z\|^2} \bigg\vert^2.
\end{align*}
    The required bound follows from the above inequality together with the bounds given in (\ref{a2BC}) and (\ref{a3BC}).

    The result is sharp for the function $G$ given by
\begin{equation}\label{extB}
     D G(z) = I \exp \int_{0}^{T_{u}(z)} \frac{\Phi(i t) - 1}{t} dt, \quad z \in \mathbb{B} ,\;\; \| u\| =1.
\end{equation}
  Clearly, $ (D G (z))^{-1} ( D^2 G(z) (z^2) + D G(z) (z) ) \in \mathcal{M}_\Phi$ and
    $$\frac{D^3  G(0) (z^3)}{3!} = - \frac{1}{6} \left( \frac{\Phi''(0)}{2} + (\Phi'(0))^2 \right) (l_u (z))^2 z   \;\; \text{and} \;\;  \frac{D^2  G (0) (z^2)}{2!}= \frac{i \Phi'(0)}{2} l_u(z) z, $$
     which immediately gives
\begin{align*}
      \frac{l_z ( D^2  G(0) (z^2))  }{2!} &= \frac{i \Phi'(0)}{2} l_u(z) \| z \|
\end{align*}
    and
   $$     \frac{l_z (D^3  G(0) (z^3)) \| z \| }{3!} = - \frac{1}{6} \left( \frac{\Phi''(0)}{2} + (\Phi'(0))^2 \right) (l_u (z))^2 \| z\|^2. $$
    Taking $z = r u$ $(0 < r <1)$, we get
\begin{equation}\label{cftbC}
   \frac{l_z (D^3  G (0) ( z^3) ) }{3! \| z \|^3}  = - \frac{1}{6} \left( \frac{\Phi''(0)}{2} + (\Phi'(0))^2 \right) \;\; \text{and} \;\; \frac{l_z (D^2  G(0) ( z^2) )}{2! \|z \|^2} = \frac{i \Phi'(0)}{2}.
\end{equation}
    According to the above equations, we have
\begin{align*}
   \bigg\vert \bigg( \frac{l_z (D^3  G (0) ( z^3) ) }{3! \| z \|^3} \bigg)^2 - \bigg(\frac{l_z (D^2  G (0) ( z^2) )}{2! \|z \|^2} \bigg)^2   \bigg\vert =
   \frac{(\Phi'(0))^2}{4} + \frac{1}{36} \bigg( \frac{\Phi''(0)}{2} + (\Phi'(0))^2  \bigg)^2,
\end{align*}
   which establishes the sharpness of the bound and completes the proof.
\end{proof}

\begin{theorem}\label{thm2C}
   Let  $g \in \mathcal{H}(\mathbb{B}, \mathbb{C})$ with $g(0)=1$, $g(z) \neq 0$, $z\in \mathbb{B}$ and suppose that $G(z) =  z g(z)$. If $D G(z))^{-1} ( D^2 G(z) (z^2) + D G(z) (z)  \in \mathcal{M}_\Phi$ such that $\Phi$ satisfy
   $$ 2 \Phi'(0) - 2 (\Phi'(0))^2  \leq \Phi''(0) \leq 4 (\Phi'(0))^2 - 2 \Phi'(0),$$
   then
   $$ \vert 2 a_2^2 a_3  - a_3^2 - 2 a_2^2 + 1 \vert \leq 1 + 2 ( \Phi'(0))^2 + \frac{ ( \Phi'(0) )^2}{4} \bigg( \frac{\Phi''(0)}{2 \Phi'(0)} -3 \Phi'(0) \bigg)  \bigg( \frac{\Phi''(0)}{2 \Phi'(0)} + \Phi'(0) \bigg), $$
   where
\begin{align*}
   a_3 = \frac{ l_z (D^3 G(0) (z^3))}{3! \vert\vert z \vert\vert^3} \;\; \text{and} \;\;  a_2 &= \frac{ l_z (D^2 G(0) (z^2))}{2! \vert\vert z \vert\vert^2}.
\end{align*}
   The bound is sharp.
\end{theorem}
\begin{proof}
     Since $ 2 \Phi'(0) \leq 4 (\Phi'(0))^2 - \Phi''(0)  $, inequality (\ref{FSBC}) gives
\begin{equation}\label{FS2BC}
\begin{aligned}
   \bigg\vert  \frac{ l_z (D^3 G(0) (z^3))}{3! \vert\vert z \vert\vert^3} -  2 \bigg(\frac{ l_z (D^2 G(0) (z^2))}{2! \vert\vert z \vert\vert^2} \bigg)^2 \bigg\vert
   \leq \frac{ \Phi'(0) }{6}   \bigg( 2  \Phi'(0) - \frac{1}{2} \frac{\Phi''(0)}{\Phi'(0)}  \bigg)
\end{aligned}
\end{equation}
    for $ z \in \mathbb{B}\setminus \{0 \}.$ Also, since $\Phi$ satisfy $ \Phi''(0) + 2 (\Phi'(0))^2 \geq 2 \Phi'(0) $, therefore by (\ref{FSBC}), we have
\begin{equation}\label{a3BC2}
     \bigg\vert  \frac{ l_z (D^3 G(0) (z^3))}{3! \vert\vert z \vert\vert^3} \bigg\vert \leq \frac{ \Phi'(0)}{6}   \bigg( \frac{1}{2} \frac{\Phi''(0)}{\Phi'(0)} +  \Phi'(0) \bigg), \quad z \in \mathbb{B}\setminus \{0\}.
\end{equation}
    Also, we have
\begin{align*}
   \vert 2 a_2^{2} a_3 - 2 a_2^2-  a_3^2+1 \vert  \leq  1 + 2 \vert a_2 \vert^2 + \vert a_3\vert \vert a_3 - 2 a_2^2\vert.
\end{align*}
    The required bound follows directly from the above inequality along with the bounds given in (\ref{a2BC}) and (\ref{a3BC2}), and the bound of $\vert a_3 - 2 a_2^2\vert$ given by (\ref{FS2BC}).

   Equality case holds for the function $G(z)$ defined in (\ref{extB}) as for this function, we have $a_2=  \frac{i \Phi'(0)}{2}$, $a_3 =- \frac{1}{6} \left( \frac{\Phi''(0)}{2} + (\Phi'(0))^2 \right) $ and hence
   $$2 a_2^{2} a_3 - 2 a_2^2-  a_3^2+1 = 2 ( \Phi'(0))^2 + \frac{ ( \Phi'(0) )^2}{12} \bigg( \frac{\Phi''(0)}{2 \Phi'(0)} -3 \Phi'(0) \bigg)  \bigg( \frac{\Phi''(0)}{2 \Phi'(0)} + \Phi'(0) \bigg) + 1, $$
   which establish the sharpness of the result.
\end{proof}
\begin{theorem}\label{ThmUn1C}
    Let $g \in \mathcal{H}(\mathbb{U}^n, \mathbb{C})$ with $g(0)=1$, $g(z) \neq 0$, $z\in \mathbb{U}^n$ and suppose that  $G(z) =  z g(z)$. If $(D G(z))^{-1} ( D^2 G(z) (z^2) + D G(z) (z) \in \mathcal{M}_\Phi$ such that $\Phi$ satisfies
    $$ \lvert \Phi''(0) + 2 (\Phi'(0))^2 \rvert \geq 2 \Phi'(0) ,$$
   then
\begin{equation}\label{res}
\begin{aligned}
\left.
\begin{array}{ll}
   \bigg\| \bigg(  & \dfrac{ D^3 G(0) (z^3)}{3! } \bigg)^2  - \bigg( \dfrac{D^2 G(0) (z^2)}{2! } \bigg)^2 \bigg\|  \\
     & \leq  \dfrac{ ( \Phi'(0) )^2  \| z\|^6}{36} \bigg( \dfrac{1}{2} \dfrac{\Phi''(0)}{\Phi'(0)} +  \Phi'(0)  \bigg)^2  +\dfrac{\left( \Phi'(0) \right)^2 \| z \|^4}{4} , \quad z \in \mathbb{U}^n.
\end{array}
\right\}
\end{aligned}
\end{equation}
   The bound is sharp.
\end{theorem}
\begin{proof}
      For $z\in \mathbb{U}^n \setminus \{0\}$ and $z_0 = \frac{z}{\| z\|}$, define $h_k : \mathbb{U} \rightarrow \mathbb{C}$ such that
\begin{equation}\label{hkzeta}
   h_k (\zeta) =
\left\{
\begin{array}{ll}
     \dfrac{\zeta z_k}{p_k (\zeta z_0) \| z_0 \|}, & \zeta \neq 0,\\
     1 , & \zeta =0,
\end{array}
\right.
\end{equation}
    where $p(z) = (D G(z))^{-1} ( D^2 G(z) (z^2) + D G(z) (z))$ and $k$ satisfies $\vert z_k \vert = \| z \| = \max_{1\leq j \leq n} \{ z_j \}$.  By (\ref{useC}), we have
    $$ h_k(\zeta) =   \frac{ D^2 g(\zeta z_0) ((\zeta z_0)^2 ) + 3 D g(\zeta z_0)(\zeta z_0) + g(\zeta z_0) }{ g(\zeta z_0)  + D g(\zeta z_0) ( \zeta z_0)} ,   $$
    or, equivalently
    $$ h_k(\zeta)  (  g(\zeta z_0)  + D g(\zeta z_0) ( \zeta z_0))  = D^2 g(\zeta z_0) ((\zeta z_0)^2 ) + 3 D g(\zeta z_0)(\zeta z_0) + g(\zeta z_0).$$
    Comparison of same homogeneous expansions in the Taylor series expansions in terms of $\zeta$ yield
\begin{equation}\label{11C}
     h_k'(0) = 2 D g(0)(z_0).
\end{equation}
    Furthermore, from $G(z_0) = z_0 g(z_0)$, we have
\begin{equation}\label{111C}
    \frac{D^2 G_k (0)(z_0^2)}{2!} = D g(0) (z_0) \frac{z_j}{\| z \|}.
\end{equation}
   Combining (\ref{11C}) and (\ref{111C}) with the fact $\vert h'_k(0) \vert \leq  \Phi'(0)$ gives
   $$ \bigg\vert \frac{D^2 G_k (0)(z_0^2)}{2!}  \frac{\| z \|}{z_j}\bigg\vert \leq \frac{\Phi'(0)}{2}.$$
   If $z_0 \in \partial \mathbb{U}^n$, then we get
\begin{equation*}
    \bigg\vert \frac{D^2 G_k (0)(z_0^2)}{2!} \bigg\vert \leq \frac{\Phi'(0)}{2}.
\end{equation*}
   Since
    $$  \frac{D^2 G_k (0)(z_0^2)}{2!}, \quad k=1,2,\cdots n $$
   are holomorphic functions on  $\overline{\mathbb{U}}^n$, therefore by the maximum modulus theorem of holomorphic functions on the unit polydisc, we have
   $$ \bigg\vert \frac{D^2 G_k (0)(z_0^2)}{2!} \bigg\vert \leq \frac{\Phi'(0)}{2} , \quad z_0 \in \mathbb{U}^n,k=1,2,\cdots n .   $$
   That is
\begin{equation}\label{a2UnC}
    \bigg\vert \frac{D^2 G_k (0)(z^2)}{2!} \bigg\vert \leq \frac{\Phi'(0)  \|z\|^2}{2} , \quad z \in \partial\mathbb{U}^n,k=1,2,\cdots n .
\end{equation}
   For $\lambda \in \mathbb{C}$, Xu et al. \cite[Theorem 3.2]{XuC} established that
\begin{equation}\label{FSUnC}
\begin{aligned}
\left.
\begin{array}{ll}
    \bigg\vert   \dfrac{ D^3 G_k(0) (z^3)}{3! }  -  &\lambda \dfrac{1}{2} D^2  G_k (0)  \bigg( z, \dfrac{D^2 G(0) (z^2)}{2!} \bigg) \bigg\vert \\
   & \leq \dfrac{\vert \Phi'(0) \vert \| z\|^3}{6} \max \bigg\{ 1  ,\left\lvert \dfrac{1}{2} \dfrac{\Phi''(0)}{\Phi'(0)} + \bigg(1 - \dfrac{3}{2} \lambda \bigg)  \Phi'(0)  \right\rvert \bigg\} .
\end{array}
\right\}
\end{aligned}
\end{equation}
    Since  $\lvert \Phi''(0) + 2 (\Phi'(0))^2 \rvert \geq 2 \Phi'(0)$, therefore, from (\ref{FSUnC}), we get
\begin{equation}\label{a3UnC}
    \bigg\vert  \frac{ D^3 G_k(0) (z^3)}{3! }   \bigg\vert  \leq \frac{ \Phi'(0)  \| z\|^3}{6}   \left( \frac{1}{2} \frac{\Phi''(0)}{\Phi'(0)} +   \Phi'(0) \right)
\end{equation}
    for $z \in \mathbb{U}^n $ and $k = 1,2, \cdots n.$ Using the bounds from (\ref{a2UnC}) and (\ref{a3UnC}), we have
\begin{align*}
   \bigg\vert \bigg(  \frac{ D^3 G_k(0) (z^3)}{3! } \bigg)^2  - \bigg( & \frac{D^2 G_k(0) (z^2)}{2! }  \bigg)^2  \bigg\vert  \\
      &\leq  \frac{ ( \Phi'(0) )^2  \| z\|^6}{36} \bigg( \frac{1}{2} \frac{\Phi''(0)}{\Phi'(0)} +  \Phi'(0)  \bigg)^2  +\frac{\left( \Phi'(0) \right)^2 \| z \|^4}{4}
\end{align*}
   for $z \in \mathbb{U}^n$  and $k=1,2, \cdots n.$ Therefore,
\begin{align*}
   \bigg\| \bigg( & \frac{ D^3 G(0) (z^3)}{3! } \bigg)^2  - \bigg( \frac{D^2 G(0) (z^2)}{2! } \bigg)^2 \bigg\| \\
      &\leq  \frac{ ( \Phi'(0) )^2  \| z\|^6}{36} \bigg( \frac{1}{2} \frac{\Phi''(0)}{\Phi'(0)} +  \Phi'(0)  \bigg)^2  +\frac{\left( \Phi'(0) \right)^2 \| z \|^4}{4} , \quad z \in \mathbb{U}^n ,
\end{align*}
    which is the required bound.

   To prove the sharpness of the bound, consider the function $G$ given by
\begin{equation}\label{extUnC}
     D G(z) = I \exp \int_0^{z_1} \frac{\Phi(i t)-1}{t} dt.
\end{equation}
    It can be showed that $(D G(z))^{-1} ( D^2 G(z) (z^2) + D G(z) (z) ) \in \mathcal{M}_\Phi$  and for $z =(r,0,\cdots,0)'$ in (\ref{extUnC}), the equality case holds in (\ref{res}).
\end{proof}
\begin{theorem}\label{bndUnC}
    Let $g \in \mathcal{H}(\mathbb{U}^n, \mathbb{C})$, $g(0)=1$, $g(z)\neq 0$, $z\in \mathbb{U}^n$   and suppose that $G(z) =  z g(z)$.     If $(D G(z))^{-1} ( D^2 G(z) (z^2) + D G(z) (z) ) \in \mathcal{M}_\Phi$ and $\Phi$ satisfy
   $$ 2 \Phi'(0) - 2 (\Phi'(0))^2  \leq \Phi''(0) \leq 4 (\Phi'(0))^2 - 2 \Phi'(0),$$
   then
\begin{align*}
     \| 2 a_2^2 a_3 & - a_3^2 - 2 a_2^2 + 1 \|   \\
     & \leq 1 +  \frac{( \Phi'(0) )^2 \| z\|^6}{36}   \bigg( 2 \Phi'(0) - \frac{1}{2} \frac{\Phi''(0)}{\Phi'(0)} \bigg)    \bigg( \frac{1}{2} \frac{\Phi''(0)}{\Phi'(0)} +   \Phi'(0) \bigg)  + \frac{( \Phi'(0) )^2 \| z\|^4}{2},
\end{align*}
    where
\begin{align*}
       a_3 = \frac{ D^3 G(0) (z^3)}{3! } \quad \text{and}  \quad a_2^2 =  \frac{1}{2} D^2  G (0)  \bigg( z, \frac{D^2 G(0) (z^2)}{2!}\bigg).
\end{align*}
   The bound is sharp.
\end{theorem}
\begin{proof}
    Since $G(z) = z g(z)$, we have
    $$  \frac{1}{2} D^2  G_k (0)  \bigg( z_0, \frac{D^2 G(0) (z_0^2)}{2!}\bigg) \frac{ z_k}{\| z \|}  = \bigg( \frac{D^2 G_k(0) (z_0^2)}{2!} \bigg)^2, \quad k=1,2,\cdots n , $$
    where $z_0 = \frac{z_k}{\| z\|}$ and $k$ satisfies $\vert z_k \vert = \| z \| = \max_{ 1 \leq j \leq n}\{ \vert z_j \vert \}$ (see \cite{XuLiu}).
    If $z_0 \in \partial \mathbb{U}^n$, then
    $$ \bigg\vert  \frac{1}{2} D^2  G_k (0)  \bigg( z_0, \frac{D^2 G(0) (z_0^2)}{2!}\bigg)  \bigg\vert = \bigg\vert \frac{D^2 G_k(0) (z_0^2)}{2!} \bigg\vert^2. $$
   Considering the function $h_k(\zeta)$ given in (\ref{hkzeta}) and following the same methodology as in the proof of Theorem \ref{ThmUn1C}, we obtain (\ref{a2UnC}), which together with the above relation yields
\begin{equation}\label{a22Unc}
   \bigg\vert  \frac{1}{2} D^2  G_k (0)  \bigg( z_0, \frac{D^2 G(0) (z_0^2)}{2!}\bigg)  \bigg\vert  \leq \frac{(\Phi'(0))^2 \|z \|^4}{4}.
\end{equation}
    Also, since $\Phi$ satisfy $ 2 \Phi'(0) \leq 4 (\Phi'(0))^2 -  \Phi''(0)$, therefore from (\ref{FSUnC}), we obtain
\begin{equation}\label{FS2UnC}
\begin{aligned}
\left.
\begin{array}{ll}
    \bigg\vert  \dfrac{ D^3 G_k(0) (z^3)}{3! }  - &  D^2  G_k (0)  \bigg( z, \dfrac{D^2 G(0) (z^2)}{2!} \bigg) \bigg\vert \\
   & \leq \dfrac{\vert \Phi'(0) \vert \| z\|^3}{6}   \left(  2  \Phi'(0) - \dfrac{1}{2} \dfrac{\Phi''(0)}{\Phi'(0)} \right) , \;\; k= 1,2, \cdots n.
\end{array}
\right\}
\end{aligned}
\end{equation}
   Thus, from (\ref{a3UnC}),  (\ref{a22Unc}) and (\ref{FS2UnC}), we have
\begin{align*}
     \bigg\vert & 1 + D^2  G_k (0)  \bigg( z, \frac{D^2 G(0) (z^2)}{2!}\bigg) \bigg(\frac{ D^3 G_k(0) (z^3)}{3! } \bigg) -  D^2  G_k (0)  \bigg( z, \frac{D^2 G(0) (z^2)}{2!}\bigg)\\
      & - \bigg( \frac{ D^3 G_k(0) (z^3)}{3! } \bigg)^2 \bigg\vert  \\
    & \leq  1 + \bigg\vert \frac{ D^3 G_k(0) (z^3)}{3! } \bigg\vert \bigg\vert \frac{ D^3 G_k(0) (z^3)}{3! } - D^2  G_k (0)  \bigg( z, \frac{D^2 G(0) (z^2)}{2!}\bigg) \bigg\vert \\
    &\;\; \;\;\;\;\; + \bigg\vert D^2  G_k(0)  \bigg( z, \frac{D^2 G(0) (z^2)}{2!}\bigg) \bigg\vert \\
    & \leq  1 +  \frac{( \Phi'(0) )^2 \| z\|^6}{36}   \bigg( 2 \Phi'(0) - \frac{1}{2} \frac{\Phi''(0)}{\Phi'(0)} \bigg)    \bigg( \frac{1}{2} \frac{\Phi''(0)}{\Phi'(0)} +   \Phi'(0) \bigg)  + \frac{( \Phi'(0) )^2 \| z\|^4}{2}
\end{align*}
    for $ z\in \mathbb{U}^n$ and $k = 1,2 , \cdots n.$
    Therefore
\begin{align*}
     \bigg\| 1 &+ D^2  G (0)  \bigg( z, \frac{D^2 G(0) (z^2)}{2!}\bigg) \bigg(\frac{ D^3 G(0) (z^3)}{3! } \bigg) -  D^2  G (0)  \bigg( z, \frac{D^2 G(0) (z^2)}{2!}\bigg)\\
      & - \bigg( \frac{ D^3 G(0) (z^3)}{3! } \bigg)^2 \bigg\|  \\
    & \leq  1 +  \frac{( \Phi'(0) )^2 \| z\|^6}{36}   \bigg( 2 \Phi'(0) - \frac{1}{2} \frac{\Phi''(0)}{\Phi'(0)} \bigg)    \bigg( \frac{1}{2} \frac{\Phi''(0)}{\Phi'(0)} +   \Phi'(0) \bigg)  + \frac{( \Phi'(0) )^2 \| z\|^4}{2},
\end{align*}
     which is the required bound.

     Sharpness of the bound can be seen from the function $G(z)$ defined in (\ref{extUnC}) by taking $z = (r,0,\cdots,0)$, which completes the proof.
\end{proof}
     Note that if $g \in \mathcal{H}(\mathbb{B})$ and $(D g(z))^{-1} ( D^2 g(z) (z^2) + D g(z) (z) ) \in \mathcal{M}_\Phi$, then various choices of $\Phi$ give different subclasses of holomorphic mappings. For instance, when $\Phi(z) = (1 + (1 - 2\alpha)z)/(1-z)$ and $\Phi(z) = (1 + z)/(1-z)$, we easily obtain $ g \in \mathcal{K}_\alpha(\mathbb{B})$ and $g \in \mathcal{K}(\mathbb{B})$, respectively. For these classes, Theorem \ref{thm1C} to Theorem \ref{bndUnC} yield the following results.
\begin{corollary}\label{crl1C}
      Let $g\in \mathcal{H}( \mathbb{B}, \mathbb{C})$ and $G(z) = z g(z) \in \mathcal{C}_\alpha(\mathbb{B})$. Then the following inequality holds:
\begin{align*}
       \bigg\vert \bigg( \frac{ l_z (D^2 G(0) (z^2))}{2! \vert\vert z \vert\vert^2} \bigg)^2  -  & \bigg(\frac{ l_z (D^3 G(0) (z^3))}{3! \vert\vert z \vert\vert^3} \bigg)^2 \bigg\vert \\
       &\leq  \frac{2 (1-\alpha)^2 (2 \alpha^2 - 6 \alpha +9)}{9}, \;\;\;  l_z \in T_z ,\; z\in \mathbb{B}\setminus\{0\} .
\end{align*}
      If $\mathbb{B} = \mathbb{U}^n$ and $X = \mathbb{C}^n$, then
\begin{equation}\label{equiv1}
\begin{aligned}
\left.
\begin{array}{ll}
     \bigg\| \bigg(  \dfrac{ D^3 G(0) (z^3)}{3! } \bigg)^2  & -  \bigg( \dfrac{D^2 G(0) (z^2)}{2! }  \bigg)^2  \bigg\| \\
      & \leq  (1 - \alpha)^2 \|z\|^4 + \dfrac{ ( 2 \alpha^2 - 5 \alpha + 3)^2 \|z \|^6 }{9} , \; z \in \mathbb{U}^n.
\end{array}
\right\}
\end{aligned}
\end{equation}
       All these bounds are sharp.
\end{corollary}
\begin{remark}
   In case of  $n=1$, $\mathbb{B}=\mathbb{U}$ and (\ref{equiv1}) reduces to the following:
\begin{equation*}
     \bigg\vert \bigg(  \frac{  G^{(3)}(0) }{3! } \bigg)^2  -  \bigg( \frac{ G''(0)}{2! }  \bigg)^2  \bigg\vert \leq \frac{2 (1-\alpha)^2 (2 \alpha^2 - 6 \alpha +9)}{9} ,
\end{equation*}
    which is equivalent to Theorem C.
\end{remark}
\begin{corollary}
      Let $g\in \mathcal{H}( \mathbb{B}, \mathbb{C})$ and $G(z) = z g(z) \in \mathcal{C}_\alpha(\mathbb{B})$. Then for $\alpha \in [0,1/2]$, the following sharp bound holds:
      $$   \vert 2 a_2^2 a_3  - a_3^2 - 2 a_2^2 + 1 \vert \leq \frac{8 \alpha^4 - 34 \alpha^3 + 71 \alpha^2 - 72 \alpha +36}{9}, \quad z \in \mathbb{B}\setminus\{0\},$$
      where
\begin{align*}
   a_3 = \frac{ l_z (D^3 G(0) (z^3))}{3! \vert\vert z \vert\vert^3} \quad\text{and} \quad  a_2 &= \frac{ l_z (D^2 G(0) (z^2))}{2! \vert\vert z \vert\vert^2}.
\end{align*}
\end{corollary}
\begin{corollary}\label{crl2}
      Let $g \in \mathcal{H}(\mathbb{U}^n, \mathbb{C})$ and $G(z) = z g(z) \in \mathcal{C}_\alpha(\mathbb{U}^n)$. Then for $\alpha \in [0,1/2]$, the following sharp inequality holds:
\begin{equation}
        \| 2 a_2^2 a_3  - a_3^2 - 2 a_2^2 + 1 \| \leq  1 + 2 \|z \|^4 (1 - \alpha)^2 + \frac{ (1 - \alpha)^2 (9 - 18 \alpha + 8 \alpha^2) \|z \|^6}{9}
\end{equation}
      for $z\in \mathbb{U}^n$, where
\begin{align*}
     a_3 = \frac{ D^3 G(0) (z^3)}{3! } \quad \text{and}  \quad a_2^2 =  \frac{1}{2} D^2  G (0)  \bigg( z, \frac{D^2 G(0) (z^2)}{2!}\bigg).
\end{align*}
\end{corollary}
\begin{remark}
    When $n=1$, Corollary \ref{crl2} is equivalent to Theorem D.
\end{remark}
     In particular, for $\alpha =0$, we obtain the following results for the class $\mathcal{C}$ in higher dimensions.
\begin{corollary}
     Let $g\in \mathcal{H}( \mathbb{B}, \mathbb{C})$ and $G(z) = z g(z) \in \mathcal{C}(\mathbb{B})$. Then the following holds:
\begin{align*}
       \bigg\vert \bigg( \frac{ l_z (D^2 G(0) (z^2))}{2! \vert\vert z \vert\vert^2} \bigg)^2  -   \bigg(\frac{ l_z (D^3 G(0) (z^3))}{3! \vert\vert z \vert\vert^3} \bigg)^2 \bigg\vert
       \leq  2, \;\;\;  l_z \in T_z ,\; z\in \mathbb{B}\setminus\{0\}.
\end{align*}
     If $\mathbb{B} = \mathbb{U}^n$ and $X = \mathbb{C}^n$, then
\begin{equation}\label{equvl2}
     \bigg\| \bigg(  \frac{ D^3 G(0) (z^3)}{3! } \bigg)^2  -  \bigg( \frac{D^2 G(0) (z^2)}{2! }  \bigg)^2  \bigg\|  \leq   \|z\|^4 + \|z \|^6 , \;\;\; z \in \mathbb{U}^n.
\end{equation}
   All these bounds are sharp.
\end{corollary}
\begin{remark}
     For $n=1$, (\ref{equvl2}) is equivalent to Theorem A.
\end{remark}
\begin{corollary}
      Let $g\in \mathcal{H}( \mathbb{B}, \mathbb{C})$ and $G(z) = z g(z) \in \mathcal{C}(\mathbb{B})$. Then the following sharp bound holds:
      $$   \vert 2 a_2^2 a_3  - a_3^2 - 2 a_2^2 + 1 \vert \leq 4, \quad z \in \mathbb{B}\setminus\{0\},$$
      where
\begin{align*}
   a_3 = \frac{ l_z (D^3 G(0) (z^3))}{3! \vert\vert z \vert\vert^3} \quad\text{and} \quad  a_2 &= \frac{ l_z (D^2 G(0) (z^2))}{2! \vert\vert z \vert\vert^2}.
\end{align*}
\end{corollary}
\begin{corollary}\label{crl6}
      Let $g \in \mathcal{H}(\mathbb{U}^n, \mathbb{C})$ and $G(z) = z g(z) \in \mathcal{C}(\mathbb{U}^n)$. Then for $\alpha \in [0,1/2]$, the following sharp estimation holds:
      $$  \vert 2 a_2^2 a_3  - a_3^2 - 2 a_2^2 + 1 \vert \leq  1 + 2 \|z \|^4  + \|z \|^6 $$
      for $z\in \mathbb{U}^n$, where
\begin{align*}
     a_3 = \frac{ D^3 G(0) (z^3)}{3! } \quad \text{and}  \quad a_2^2 =  \frac{1}{2} D^2  G (0)  \bigg( z, \frac{D^2 G(0) (z^2)}{2!}\bigg).
\end{align*}
\end{corollary}
\begin{remark}
   when $n=1$, Corollary \ref{crl6} is equivalent to Theorem B.
\end{remark}

\section*{Declarations}
\subsection*{Funding}
The work of the Surya Giri is supported by University Grant Commission, New-Delhi, India  under UGC-Ref. No. 1112/(CSIR-UGC NET JUNE 2019).
\subsection*{Conflict of interest}
	The authors declare that they have no conflict of interest.
\subsection*{Author Contribution}
    Each author contributed equally to the research and preparation of manuscript.
\subsection*{Data Availability} Not Applicable.
\noindent

\end{document}